\documentclass[12pt, reqno]{amsart}

\usepackage{fullpage}
\usepackage{amsmath}
\usepackage{amssymb}

\newcommand{\N}{\mathbb{N}}

\newcommand{\R}{\mathbb{R}}
\newcommand{\C}{\mathbb{C}}

\newcommand{\B}{\mathbb{B}}
\newcommand{\D}{\mathbb{D}}
\renewcommand{\S}{\mathbb{S}}
\newcommand{\mc}{\mathcal}
\newcommand{\mb}{\mathbf}

\newcommand{\rg}{\operatorname{rg}}

\renewcommand{\Re}{\operatorname{Re}}

\newtheorem{lemma}{Lemma}[section]
\newtheorem{theorem}[lemma]{Theorem}
\newtheorem{corollary}[lemma]{Corollary}
\newtheorem{proposition}[lemma]{Proposition}

\theoremstyle{remark}
\newtheorem{remark}[lemma]{Remark}

\theoremstyle{definition}
\newtheorem{definition}[lemma]{Definition}

\numberwithin{equation}{section}

\title{Spectral theory and self-similar blowup in wave equations}
\author{Roland Donninger}
\address{Universität Wien, Fakultät für Mathematik,
  Oskar-Morgenstern-Platz 1, 1090 Vienna, Austria}
\email{roland.donninger@univie.ac.at}
\thanks{This work was supported by the Austrian Science Fund FWF,
  Project P 34560: ``Stable blowup in supercritical wave equations''.}

\begin{document}
\maketitle
\begin{abstract}
  This is an expository article that describes the
  spectral-theoretic aspects in the study of the stability of
  self-similar blowup for nonlinear wave equations. The
  linearization near a self-similar solution leads to a genuinely
  nonself-adjoint operator which is difficult to
  analyze. The main goal of this article is to provide an accessible
  account to the only known method that is capable of providing
  sufficient spectral information to complete the stability
  analysis. The exposition is based on
  a mini course given at the
  \emph{Summer School on Geometric Dispersive PDEs} in Obergurgl,
  Austria, in September 2022.
\end{abstract}

\section{Introduction}

Nonlinear wave equations play a fundamental role in many branches of
the natural sciences and mathematics. Probably the most famous
examples in physics are the
Einstein equation of general relativity and the Yang-Mills equations of
particle physics. What all of these fundamental equations have in
common is the fact that they are \emph{energy-supercritical} (in the
case of
Yang-Mills in spatial dimensions larger than four). This
means that the known conserved quantities (most notably the energy) are not
strong enough to control the evolution. As a result, the mathematical
understanding of large-data evolutions is still embarrassingly
poor. In many cases, however, there exist self-similar solutions
and one may learn something about the general large-data behavior by looking at
perturbations of these large but special solutions. This approach is promising because
it allows one to employ
perturbative techniques in a large-data regime that is otherwise
inaccessible to rigorous mathematical analysis. Such a perturbative treatment
involves a number of interesting spectral-theoretic aspects that are
at the center of this article.

\subsection{Wave maps}

For the purpose of this exposition we will not discuss nonlinear wave
equations in any kind of generality but rather focus on a particular
example:
the classical \emph{SU(2)-sigma model} from particle physics, also
known as the
\emph{wave maps equation}, which constitutes the simplest and prototypical example of a \emph{geometric
wave equation}. The methods we discuss, however, have a much broader
scope and we mention applications to other problems in the end. In order to introduce the model, 
we consider maps $U: \R^{1,3}\to\S^3\subset \R^4$, where $\R^{1,3}$
denotes the $(1+3)$-dimensional Minkowski space. Then $U$ is
called a wave map if it satisfies
the partial differential equation
\begin{equation}
  \label{eq:wm}
  \partial^\mu\partial_\mu U+\langle \partial_\mu U, \partial^\mu
  U\rangle_{\R^4}U=0.
\end{equation}
Here, we employ standard relativistic notation with  Einstein's
summation convention in force\footnote{That is to say, we number the
  slots of a function on $\R^{1,3}$ from $0$ to $3$ where the $0$-th
  slot holds the time variable. The partial derivative with respect
  to the $\mu$-th slot is denoted by $\partial_\mu$ and we write
  $\partial^0:=-\partial_0$ as well as $\partial^j:=\partial_j$ for
  $j\in\{1,2,3\}$. Furthermore, indices that come in pairs of
  subscripts and superscripts get summed over implicitly. Greek (spacetime) indices run
  from $0$ to $3$ and latin (spatial) indices run from $1$ to $3$.} and $\langle\cdot,\cdot\rangle_{\R^4}$
denotes the Euclidean inner product on $\R^4$.
The wave maps equation arises as the Euler-Lagrange equation of the
action functional
\begin{equation}
  \label{eq:action}
  U\mapsto \int_{\R^{1,3}}\langle \partial^\mu U, \partial_\mu
  U\rangle_{\R^4}
\end{equation}
under the constraint that $U(t,x)\in\S^3$ for all $(t,x)\in
\R^{1,3}$. Note that without the constraint, the Euler-Lagrange
equation associated to Eq.~\eqref{eq:action} is the standard free wave equation
$\partial^\mu\partial_\mu U=0$. In this sense, wave maps are natural generalizations of
solutions to the  wave
equation when the unknown takes values in
the sphere. In place of Minkowski space and the three-sphere, one may also consider more general
  manifolds by adapting the functional
  \eqref{eq:action} accordingly. This shows that the wave maps action
  is a rich source for interesting and natural geometric wave equations. In this exposition, for the sake of concreteness, we restrict
  ourselves to maps from $\R^{1,3}$ to $\S^3$.
We remark in passing that in more traditional
 notation, Eq.~\eqref{eq:wm} would read
\[ -\partial_t^2 U(t,x)+\Delta_x U(t,x)=\left(\langle\partial_tU(t,x),
  \partial_tU(t,x)\rangle_{\R^4}-\sum_{j=1}^3
  \langle\partial_{x^j}U(t,x),
  \partial_{x^j}U(t,x)\rangle_{\R^4}\right)U(t,x) \]
but in this form, the underlying geometric structure is severely obscured.

\subsection{Corotational wave maps and singularity formation}
The most basic question concerns the existence of smooth solutions to Eq.~\eqref{eq:wm}.
For the sake of simplicity, we further restrict our attention to the special
class of \emph{corotational maps} which are of the
form
\[ U(t,x)=
  \begin{pmatrix}
    \sin(u(t,|x|))\frac{x}{|x|} \\
    \cos(u(t,|x|))
  \end{pmatrix}
\]
for an auxiliary function $u: \R\times [0,\infty)\to\R$.
This ansatz turns out to be compatible with the wave maps equation,
i.e., when plugging it in, we obtain the single semilinear radial wave
equation
\begin{equation}
  \label{eq:wmcor}
  \left (\partial_t^2-\partial_r^2-\frac{2}{r}\partial_r\right)u(t,r)+\frac{\sin(2u(t,r))}{r^2}=0.
\end{equation}
The principal goal is to construct global solutions and since Eq.~\eqref{eq:wmcor} is
a wave equation, the natural mathematical setting to approach this
question is to study the
\emph{Cauchy problem}, i.e., we prescribe \emph{initial data} $u(0,\cdot)$ and $\partial_0
u(0,\cdot)$ and try to construct a solution to Eq.~\eqref{eq:wmcor} with these data. However,
\[ u^T(t,r):=2\arctan(\tfrac{r}{T-t}) \]
for any $T\in \R$ solves Eq.~\eqref{eq:wmcor} on $\R\times
[0,\infty)\setminus\{(T,0)\}$, as one checks by a direct computation. At $(t,r)=(T,0)$,
$u^T$ exhibits a gradient blowup and hence, it is impossible to
construct global smooth solutions for arbitrary data. Consequently, the goal is to understand the
nature of this breakdown (or ``loss of smoothness'' or ``singularity
formation'' or ``blowup'') and its relevance for ``generic''
evolutions. More precisely, the question is whether $u^T$ can tell us
something about more general large-data evolutions, even though it is just
one particular solution. In other words, we are interested in
stability properties of $u^T$, i.e., we would like to understand
\emph{all} solutions that are close to $u^T$. We
remark that $u^T$ is a \emph{self-similar solution}, i.e., it depends on
the ratio $\frac{r}{T-t}$ only. The existence of self-similar
solutions to Eq.~\eqref{eq:wmcor} was first proved in
\cite{Sha88} and the explicit example $u^T$ was found in
\cite{TurSpe90}, see \cite{BizBie15} for higher
dimensions.
In fact, there are many more self-similar
solutions to Eq.~\eqref{eq:wmcor}, see \cite{Biz00}, but they are all
linearly unstable and hence less important for studying generic evolutions.

\section{The mode stability problem}

If the self-similar solution $u^T$ has any relevance for generic
large-data evolutions, it certainly must be stable under perturbations of
the initial data. Thus, an important mathematical goal is to prove (or
disprove) the stability of $u^T$. The most elementary form of
stability is \emph{mode stability}. The formulation of the mode
stability problem can
be given purely on the level of the differential equation and
requires no operator-theoretic framework. 

\subsection{Similarity coordinates}
In order to introduce the mode stability problem,
we start with the wave maps equation \eqref{eq:wmcor}
and switch to \emph{similarity coordinates}
\begin{equation}
  \label{eq:sim}
  \tau=-\log(T-t)+\log T,\qquad \rho=\frac{r}{T-t} 
\end{equation}
or
\[ t=T-Te^{-\tau},\qquad r=Te^{-\tau}\rho, \]
where $T>0$ is a parameter.
Then $u$ satisfies Eq.~\eqref{eq:wmcor} if and only if
$v_T(\tau,\rho):=u(T-Te^{-\tau}, Te^{-\tau}\rho)$ satisfies
\begin{equation}
  \label{eq:wmcorss}\left
    [\partial_\tau^2+2\rho\partial_\tau\partial_\rho+\partial_\tau
    -(1-\rho^2)\partial_\rho^2+\left (2\rho-\frac{2}{\rho}\right)\partial_\rho\right ]
     v_T(\tau,\rho)+\frac{\sin(2v_T(\tau,\rho))}{\rho^2}=0.
   \end{equation}
 Observe the remarkable fact that Eq.~\eqref{eq:wmcorss} is an
 autonomous equation, i.e., its coefficients do not depend on
 $\tau$. This is in fact a decisive feature of the similarity coordinates
 \eqref{eq:sim}.
 Furthermore, the parameter $T$ does not show up in
 Eq.~\eqref{eq:wmcorss}. To begin with, we will consider
 Eq.~\eqref{eq:wmcorss} in the coordinate range $\tau\geq 0$ and
 $\rho\in [0,1]$, which corresponds to the backward lightcone of the
 point $(T,0)$ in the ``physical'' coordinates $(t,r)$.
 
The blowup solution
$u^{T'}(t,r)=2\arctan(\frac{r}{T'-t})$ transforms into
\[
  v_T^{T'}(\tau,\rho):=u^{T'}(T-Te^{-\tau},Te^{-\tau}\rho)=2\arctan\left
    (\frac{\rho}{1+(\frac{T'}{T}-1)e^\tau}\right). \]
We would like to understand the stability of the family $\{v_T^{T'}:
T'>0\}$.
First, let us point out that $v_T^T$ is independent of $\tau$ whereas
nearby members of the family move
away from $v^T_T$ as $\tau$ increases. Indeed, if $T'<T$,
$v^{T'}_T(\tau,\cdot)$ develops a gradient blowup as $\tau\to\tau_*$,
where $\tau_*$ is determined by $(\frac{T'}{T}-1)e^{\tau_*}=-1$.
On the other hand, if $T'>T$, $v_T^{T'}(\tau,\rho)\to 0$ as
$\tau\to\infty$.
 By these observations, it is expected that the $\tau$-independent solution $v_T^T$ is
unstable because a generic perturbation will push it towards a
nearby member of the family. However, such a ``push'' can be compensated by
adapting $T$. Thus, the instability is ``artificial'' and
caused by the free parameter $T$ in the definition of the similarity
coordinates or, on a more fundamental level, by the time-translation
invariance of the wave maps equation. In other words, stability of the
blowup means that for any given
(small) initial perturbation of $u^1$, say, there exists a
$T$ close to $1$ that makes the
evolution in similarity coordinates with parameter $T$ converge to
$v_T^T$. This is very natural in view of the expectation that a
perturbation of a blowup solution will in general change the blowup time.

\subsection{Mode solutions}
The most elementary stability analysis consists of looking for
\emph{mode solutions}. This means that one plugs in the ansatz
\[ v_T(\tau,\rho)=v_T^T(\rho)+e^{\lambda\tau}f(\rho),\qquad \lambda\in\C \]
into Eq.~\eqref{eq:wmcorss} and linearizes in $f$. This yields the
``spectral problem''
\begin{equation}
  \label{eq:spec}-(1-\rho^2)f''(\rho)-\frac{2}{\rho}f'(\rho)+2(\lambda+1)\rho
  f'(\rho)+\frac{2\cos(2v_T^T(\rho))}{\rho^2}f(\rho)+\lambda(\lambda+1)f(\rho)=0.
\end{equation}
Clearly, if there are ``admissible'' mode solutions with
$\Re\lambda>0$, we expect the solution $v_T^T$ to be unstable.
What exactly ``admissible'' in this context means can only be answered
once one has set up the functional analytic framework to study the
wave maps evolution. For now we will restrict ourselves to smooth
solutions and we will see later that this is the correct class of
functions. Furthermore, observe that Eq.~\eqref{eq:spec} has singular points at
$\rho=0$ and $\rho=1$ and therefore, it is expected that only for
special values of $\lambda$ there will be
nontrivial solutions in $C^\infty([0,1])$. Another important fact is that
Eq.~\eqref{eq:spec} does not constitute a standard eigenvalue problem
because the spectral parameter $\lambda$ appears in the coefficient of the
derivative $f'$ as well. This is easily traced back to the fact that
the wave maps equation is second-order in time. Consequently, this
issue is not present in analogous parabolic problems where the
corresponding spectral analysis is therefore much simpler. 
Of course, the first-order term can always
be removed but the corresponding transformation depends on $\lambda$ itself. As a
consequence, it turns out that Eq.~\eqref{eq:spec} is \emph{not} a
self-adjoint Sturm-Liouville problem in disguise where standard methods from
mathematical physics would apply. We discuss this in more detail below.

We have
already argued that we expect an ``artificial'' instability of
$v_T^T$. So how does this instability show up in the context of the
spectral problem Eq.~\eqref{eq:spec}? To see this, we differentiate
the equation
\[  \left [\partial_\tau^2+2\rho\partial_\tau\partial_\rho+\partial_\tau
    -(1-\rho^2)\partial_\rho^2+\left (2\rho-\frac{2}{\rho}\right)\partial_\rho\right ]
     v_T^{T'}(\tau,\rho)+\frac{\sin(2v_T^{T'}(\tau,\rho))}{\rho^2}=0.
   \]
   with respect to $T'$ and evaluate the result at $T'=T$. This yields
 \[  \left [\partial_\tau^2+2\rho\partial_\tau\partial_\rho+\partial_\tau
    -(1-\rho^2)\partial_\rho^2+\left (2\rho-\frac{2}{\rho}\right)\partial_\rho\right ]
  v_*(\tau,\rho)+\frac{2\cos(2v_T^{T}(\rho))}{\rho^2}
  v_*(\tau,\rho)=0
\]
with
\[ v_*(\tau,\rho):=\partial_{T'}v_T^{T'}(\tau,\rho)\big
  |_{T=T'}=-\frac{2}{T}e^\tau\frac{\rho}{1+\rho^2}. \]
Observe that $v_*$ is a mode solution.
Consequently, the function $\rho\mapsto\frac{\rho}{1+\rho^2}$ solves
Eq.~\eqref{eq:spec} with $\lambda=1$ and this is the mode solution that
reflects the expected ``artificial'' instability. This observation naturally leads
to the following definition.

\begin{definition}
  \label{def:modstab}
  We say that the blowup solution $u^T$ is \emph{mode
    stable}\footnote{The experienced reader might think ahead and be worried about
    spectral multiplicities. It turns out that this is never an issue
    in the class of problems we consider here and therefore,
    Definition \ref{def:modstab} is the ``correct'' one. At this point we cannot
    even discuss multiplicities because we do not yet have a
    proper operator-theoretic framework.} if the
  existence of a nontrivial $f\in C^\infty([0,1])$ that satisfies
  Eq.~\eqref{eq:spec} necessarily implies that $\Re\lambda<0$ or $\lambda=1$.
\end{definition}

In what follows, we somewhat imprecisely call $\lambda\in \C$ an
\emph{eigenvalue} of Eq.~\eqref{eq:spec} if Eq.~\eqref{eq:spec} has a
nontrivial solution in $C^\infty([0,1])$. Accordingly, we call such a
solution an \emph{eigenfunction} of Eq.~\eqref{eq:spec}.

\section{Solution of the mode stability problem}

In this section, which is at the heart of the present exposition, we describe an approach to the mode stability problem that was developed in Irfan
Glogi\'c's PhD thesis \cite{Glo18} and first published in
\cite{CosDonGloHua16, CosDonGlo17}, building on earlier work
\cite{DonAic08, DonSchAic12, CosDonXia16} and ideas in
\cite{BizChmTab00, Biz00, Biz05}. So far, it is the only known method that can rigorously
deal with spectral problems like Eq.~\eqref{eq:spec}.

\begin{theorem}
  \label{thm:modstab}
  The blowup solution $u^T$ is mode stable.
\end{theorem}

The proof of mode stability proceeds by the following main steps.

\begin{itemize}
\item We use Frobenius' method to determine the local behavior of solutions to
  Eq.~\eqref{eq:spec} near the singular points $\rho=0$ and
  $\rho=1$.

\item By a factorization procedure inspired by supersymmetric quantum
  mechanics we ``remove'' the eigenvalue $\lambda=1$. More precisely,
  we derive a ``supersymmetric problem'', similar to Eq.~\eqref{eq:spec},
  that has the same eigenvalues as Eq.~\eqref{eq:spec} except for
  $\lambda=1$.

\item We prove that the supersymmetric problem has no eigenvalues in
  the closed complex right half-plane. To this end, we derive a
  recurrence relation for the coefficients of the power series of the
  admissible solution near $\rho=0$ and prove that the series
  necessarily diverges
  at $\rho=1$ if $\Re\lambda\geq 0$. This requires the interplay of
  techniques from the theory of difference equations and complex analysis.
  
\end{itemize}

\subsection{Fuchsian classification}

To begin with, we would like to understand better which problem we are
actually facing.
The term in Eq.~\eqref{eq:spec} involving the cosine turns out to be a
rational function. Indeed, we have
\[ 2\cos(2v_T^T(\rho))=2\cos(4\arctan(\rho))=2\frac{1-6\rho^2+\rho^4}{(1+\rho^2)^2} \]
and thus, Eq.~\eqref{eq:spec} has the (regular) singular points $0, \pm 1, \pm
i, \infty$. By switching to the independent variable $\rho^2$, the
number of singular points can be reduced to four: $0, \pm 1$, and
$\infty$. This means that Eq.~\eqref{eq:spec} is a Fuchsian
differential equation of Heun type.
The normal form for a Heun equation reads
\[ g''(z)+\left
    [\frac{\gamma}{z}+\frac{\delta}{z-1}+\frac{\epsilon}{z-a}\right]g'(z)+\frac{\alpha\beta
    z-q}{z(z-1)(z-a)}g(z)=0 \]
where $\alpha,\beta,\gamma,\delta,\epsilon,a,q\in\C$.
Around each of the singular points there exist two linearly
independent local solutions. The interesting question then is how
local solutions around different singular points are related to each
other. This is known as the \emph{connection problem} and
unfortunately, for Heun equations this problem is widely open. If we had only
three regular
singular points, we would be dealing with a \emph{hypergeometric
  differential equation} for which the connection problem was
solved in the 19th century. This indicates that the spectral problem
we are dealing with is potentially hard.

\subsection{Frobenius analysis}

Now we turn to a more quantitative analysis and first recall
Frobenius' theory for Fuchsian equations of second order. These are
equations over the complex numbers of the form
\begin{equation}
  \label{eq:Fuchs}
  f''(z)+p(z)f'(z)+q(z)f(z)=0
\end{equation}
where $p$ and $q$ are given functions and $f$ is the unknown. In the
following, we write $\D_R:=\{z\in\C: |z|<R\}$.

\begin{theorem}
  \label{thm:Frob}
  Let $R>0$ and let $p,q: \D_R\setminus\{0\}\to\C$ be
  holomorphic. Suppose that the limits
  \[ p_0:=\lim_{z\to 0}[zp(z)],\qquad q_0:=\lim_{z\to 0}[z^2q(z)] \]
  exist and let $s_\pm\in \C$ satisfy $P(s_\pm)=0$, where
  \[ P(s):=s(s-1)+p_0s+q_0 \]
is the \emph{indicial
    polynomial}.
  Let $\Re s_+\geq \Re s_-$. Then there exists a holomorphic function
  $h_+: \D_R\to \C$ with $h_+(0)=1$ and such that $f:
  \D_R\setminus (-\infty,0]\to \C$, given by $f(z)=z^{s_+}h_+(z)$,
  satisfies Eq.~\eqref{eq:Fuchs}.
 Furthermore, if
  $s_+-s_-\notin\N_0$, there exists a holomorphic function $h_-:
  \D_R\to\C$ with $h_-(0)=1$ and such that $f(z)=z^{s_-}h_-(z)$ is
  another solution of Eq.~\eqref{eq:Fuchs} on $\D_R\setminus (-\infty,0]$. 
   Finally, if $s_+-s_-\in \N_0$, there exist $c\in
  \C$ and a
  holomorphic function $h_-: \D_R\to\C$ with $h_-(0)=1$ such that
  \[ f(z)=z^{s_-}h_-(z)+cz^{s_+}h_+(z)\log z \]
  is another solution of Eq.~\eqref{eq:Fuchs} on $\D_R\setminus (-\infty,0]$.
\end{theorem}

\begin{proof}[Idea of proof]
  The idea is to plug in a generalized power series ansatz
  $z^\sigma\sum_{k=0}^\infty a_kz^k$ and to determine $\sigma$ and the
  coefficients $(a_k)_{k\in\N_0}$ by comparing powers of $z$. The
  convergence of the corresponding series is then shown by a simple
  induction. The second solution can be obtained by the reduction of order
  ansatz. We remark in passing that even in the case $s_+-s_-\in\N_0$, the $\log$ term may be
  absent but this depends on the fine
  structure and needs to be analyzed on a case-by-case basis.
  We omit the details of the proof because Theorem \ref{thm:Frob} is a classical result that
  can be found in many textbooks, see e.g.~\cite{Tes12} for a modern
  account. 
\end{proof}

Slightly re-arranged, Eq.~\eqref{eq:spec} reads
\begin{equation}
  \label{eq:specre}f''(\rho)+2\frac{1-(\lambda+1)\rho^2}{\rho(1-\rho^2)}f'(\rho)-\left
    [V(\rho)+\frac{\lambda(\lambda+1)}{1-\rho^2}\right]f(\rho)=0.
\end{equation}
with
\[ V(\rho):=2\frac{1-6\rho^2+\rho^4}{\rho^2(1-\rho^2)(1+\rho^2)^2} \]
and the indicial polynomial at $\rho=0$ reads $s(s-1)+2s-2$
with zeros $1$ and $-2$. As expected, there is only one smooth solution around
$\rho=0$ and it behaves like $\rho$. At $\rho=1$, the indicial polynomial is
given by $s(s-1)+\lambda s=0$ with zeros $0$ and $1-\lambda$. Again,
there is only one smooth solution around
$\rho=1$ if $\Re\lambda\geq 0$ (the cases $\lambda\in \{0,1\}$ require some extra care). Thus, our goal is to show that the local solution that is
smooth around $\rho=0$ is necessarily nonsmooth at $\rho=1$ if
$\Re\lambda\geq 0$ (and $\lambda\not=1$).

\subsection{Supersymmetric removal}

The case $\lambda=1$ is special and we already know that this is an
eigenvalue. In order to proceed,
it is necessary to ``remove'' it. This can be achieved by a factorization
procedure that has its origin in supersymmetric quantum mechanics
(hence the name). In our setting, the procedure is as follows. First,
we introduce an auxiliary function $g$ by $f(\rho)=p(\rho)g(\rho)$,
where we choose $p$ in such a way that the resulting equation for $g$
has no first-order derivative. Indeed, inserting the above ansatz into
Eq.~\eqref{eq:specre} yields the condition
\[ p'(\rho)=-\frac{1-(\lambda+1)\rho^2}{\rho(1-\rho^2)}p(\rho) \]
which is satisfied e.g.~by
$p(\rho)=\rho^{-1}(1-\rho^2)^{-\frac{\lambda}{2}}$.
Plugging the ansatz
\[ f(\rho)=\rho^{-1}(1-\rho^2)^{-\frac{\lambda}{2}}g(\rho) \] into
Eq.~\eqref{eq:specre} yields
\begin{equation}
  \label{eq:specg}
  g''(\rho)-V(\rho)g(\rho)=\frac{\lambda(\lambda-2)}{(1-\rho^2)^2}g(\rho).
  \end{equation}
Recall that the function $\rho\mapsto\frac{\rho}{1+\rho^2}$ solves
Eq.~\eqref{eq:specre} with $\lambda=1$. Thus,
\[ g_0(\rho):=(1-\rho^2)^{\frac12}\frac{\rho^2}{1+\rho^2} \]
satisfies
\[ g_0''(\rho)-V(\rho)g_0(\rho)=-\frac{1}{(1-\rho^2)^2}g_0(\rho). \]
Motivated by this, we rewrite Eq.~\eqref{eq:specg} as
\begin{equation}
  \label{eq:Schrod}
  g''(\rho)+\left[\frac{1}{(1-\rho^2)^2}-V(\rho)\right]g(\rho)=\frac{(\lambda-1)^2}{(1-\rho^2)^2}g(\rho).
  \end{equation}
This resembles a spectral problem for a Schr\"odinger
operator with a
ground state $g_0$.

At this point we digress and re-iterate that our mode
stability problem \emph{cannot} be reduced to studying the spectrum of
the self-adjoint realization of the 
Schr\"odinger operator in Eq.~\eqref{eq:Schrod}. The reason is
that an admissible eigenfunction of Eq.~\eqref{eq:spec} transforms
into a
solution of Eq.~\eqref{eq:Schrod} that behaves like
$(1-\rho)^{\frac{\lambda}{2}}$ near $\rho=1$. 
However, if $\Re\lambda\leq 1$, this function is not in
$L^2_w(0,1)$ with weight $w(\rho)=\frac{1}{(1-\rho^2)^2}$, which is
the natural Hilbert space for Eq.~\eqref{eq:Schrod}. As a consequence,
eigenvalues $\lambda$ with $\Re\lambda\leq 1$ are ``invisible'' in the
``self-adjoint picture'' of Eq.~\eqref{eq:Schrod}.

Nevertheless, we can employ the factorization procedure from
supersymmetric quantum mechanics, as this is in fact a pure ODE
argument that has nothing to do with operator theory. To this end, observe that $g_0$ has no zeros in $(0,1)$ and we have the factorization
\begin{align*}
  \left(\partial_\rho+\frac{g_0'(\rho)}{g_0(\rho)}\right)\left
    (\partial_\rho-\frac{g_0'(\rho)}{g_0(\rho)}\right)
  &=\partial_\rho^2-\partial_\rho\left(\frac{g_0'(\rho)}{g_0(\rho)}\right)-\frac{g_0'(\rho)^2}{g_0(\rho)^2}
    =\partial_\rho^2-\frac{g_0''(\rho)}{g_0(\rho)} \\
  &=\partial_\rho^2+\left [\frac{1}{(1-\rho^2)^2}-V(\rho)\right].
\end{align*}
Consequently, Eq.~\eqref{eq:Schrod} can be written as
\[ (1-\rho^2)^2\left(\partial_\rho+\frac{g_0'(\rho)}{g_0(\rho)}\right)
    \left [\left (\partial_\rho-\frac{g_0'(\rho)}{g_0(\rho)}\right)g(\rho)\right]=(\lambda-1)^2g(\rho). \]
The trick is now to apply the operator
$\partial_\rho-\frac{g_0'(\rho)}{g_0(\rho)}$ to this equation. In
terms of
\[ \widetilde g(\rho):=\left
    (\partial_\rho-\frac{g_0'(\rho)}{g_0(\rho)}\right)g(\rho), \]
the resulting equation reads
\[
  \left(\partial_\rho-\frac{g_0'(\rho)}{g_0(\rho)}\right)\left[(1-\rho^2)^2\left
    (\partial_\rho+\frac{g_0'(\rho)}{g_0(\rho)}\right)\widetilde
  g(\rho)\right]=(\lambda-1)^2\widetilde g(\rho). \]
Note that
\[ \left (\partial_\rho -
    \frac{g_0'(\rho)}{g_0(\rho)}\right)g_0(\rho)=0, \]
i.e., the solution that comes from the artificial instability gets
annihilated by this transformation.
Finally, we write $\widetilde
f(\rho)=\rho^{-1}(1-\rho^2)^{1-\frac{\lambda}{2}}\widetilde g(\rho)$
and the equation turns into
\begin{equation}
  \label{eq:specss}-(1-\rho^2)\widetilde
  f''(\rho)-\frac{2}{\rho}\widetilde f'(\rho)+2(\lambda+1)\rho
  \widetilde f'(\rho)+\frac{2(3-\rho^2)}{\rho^2(1+\rho^2)}\widetilde
  f(\rho)+\lambda(\lambda+1)\widetilde f(\rho)=0,
\end{equation}
which has the exact same structure as Eq.~\eqref{eq:spec} but with a
different ``potential''. Based on the above, we have the following
correspondence result.

\begin{lemma}
\label{lem:ss}
  Let $\lambda\in \C\setminus\{1\}$ and suppose that there exists a
  nontrivial $f\in C^\infty([0,1])$ that satisfies
  Eq.~\eqref{eq:spec}. Then there exists a nontrivial $\widetilde f\in
  C^\infty([0,1])$ that satisfies Eq.~\eqref{eq:specss}.
\end{lemma}

\begin{proof}
  Given $f$, we set
  \[ \widetilde
    f(\rho):=\rho^{-1}(1-\rho^2)^{1-\frac{\lambda}{2}}\left
      (\partial_\rho-\frac{2-3\rho^2-\rho^4}{\rho(1-\rho^2)(1+\rho^2)}\right)\left
      [\rho(1-\rho^2)^{\frac{\lambda}{2}}f(\rho)\right] \]
  and since
  \[
    \frac{g_0'(\rho)}{g_0(\rho)}=\frac{2-3\rho^2-\rho^4}{\rho(1-\rho^2)(1+\rho^2)}, \]
the above derivation shows that $\widetilde f$ is nontrivial (here
$\lambda\not=1$ is used) and satisfies
Eq.~\eqref{eq:specss}. The fact that $\widetilde f\in
C^\infty([0,1])$ follows by inspection because $f(\rho)$ behaves
like $\rho$ near $0$ by Frobenius' method.
\end{proof}

\subsection{Transformation to standard Heun form}

Eq.~\eqref{eq:specss} is again of Heun type. To see this,
we first observe that the indicial polynomial of Eq.~\eqref{eq:specss}
at $\rho=0$ is $s(s-1)+2s-6$ with zeros $2$ and $-3$. At $\rho=1$ we
have, as with the original equation, $s(s-1)+\lambda s$ with zeros $0$
and $1-\lambda$. In order to obtain the standard Heun form, one of the
indices at each of the singular points must equal zero.
Thus, we introduce
new variables by writing $\widetilde f(\rho)=\rho^2 \widehat f(\rho^2)$.
Then $\widetilde f$ satisfies Eq.~\eqref{eq:specss} if and only if
$\widehat f$ satisfies the Heun equation
\begin{equation}
  \label{eq:specssHeun0}
  \widehat f''(x)+\left
    (\frac{7}{2x}+\frac{\lambda}{x-1}\right)\widehat
  f'(x)+\frac14\frac{(\lambda+3)(\lambda+2)x+\lambda^2+5\lambda-2}{x(x-1)(x+1)}\widehat
  f(x)=0.
\end{equation}
The domain we are interested in is $x\in [0,1]$ (which corresponds to
$\rho\in [0,1]$). However, as will become clear below, the fact that the singularity at
$x=-1$ has the same distance from $0$ as the singularity at $x=1$
spoils our analysis. For this reason, we need to move it, which is
possible by the M\"obius transform $x\mapsto \frac{2x}{x+1}$, which
maps $0$ to $0$, $1$ to $1$, $-1$ to $\infty$, and $\infty$ to
$2$. Then $1$ is the only singularity within distance $1$ from $0$. We
note that this transformation was introduced in the present context in
\cite{Biz05}. Upon
writing
\[ \widehat f(x)=\left
    (2-\frac{2x}{x+1}\right)^{1+\frac{\lambda}{2}}g\left
    (\frac{2x}{x+1}\right), \]
we finally arrive at the Heun equation
\begin{equation}
  \label{eq:specssHeun}
  g''(z)+\left
    (\frac{7}{2z}+\frac{\lambda}{z-1}+\frac{1}{2(z-2)}\right) g'(z)
  +\frac14\frac{(\lambda+4)(\lambda+2)z-(\lambda^2+12\lambda+12)}{z(z-1)(z-2)}g(z)=0.
\end{equation}
Tracing back the above derivation, we obtain the following lemma.

\begin{lemma}
  Let $\lambda\in \C\setminus\{1\}$ and suppose that there exists a
  nontrivial
  $f\in C^\infty([0,1])$ that satisfies Eq.~\eqref{eq:spec}. Then there
  exists a nontrivial $g\in C^\infty([0,1])$ that satisfies Eq.~\eqref{eq:specssHeun}.
\end{lemma}

\subsection{The recurrence relation}

The indicial polynomial of Eq.~\eqref{eq:specssHeun} at $z=0$ is
$s(s-1)+\frac72s$ with zeros $0$ and $-\frac52$. At $z=1$ we have
$s(s-1)+\lambda s=0$ with zeros $0$ and $1-\lambda$. Thus, a solution
$g\in C^\infty([0,1])$ is holomorphic around both $z=0$ and
$z=1$. The ``next'' singularity in Eq.~\eqref{eq:specssHeun} is
at $z=2$ and thus, a solution $g\in C^\infty([0,1])$ is in fact
holomorphic on $\D_2$. Note that this line of reasoning would fail for
Eq.~\eqref{eq:specssHeun0} because of the singularity at $x=-1$.
Since the power series representation of a function that is
holomorphic on a disc converges on that very disc, we see that a
solution $g\in C^\infty([0,1])$ of Eq.~\eqref{eq:specssHeun} can be represented by a power series
centered at $z=0$ with radius of convergence at least $2$. Thus, the idea is
to insert a power series ansatz, obtain a recurrence relation for the
coefficients and then prove that the radius of convergence equals $1$
if $\Re\lambda\geq 0$. The reduction of the mode stability problem to
the convergence properties of the corresponding power series is from
\cite{Biz05}, which also provides convincing numerical evidence for mode
stability.

Concretely, from Frobenius' theory we know that
there exists a solution
\[ g(z)=\sum_{n=0}^\infty a_n z^n \]
to Eq.~\eqref{eq:specssHeun}, where the power series has radius of
convergence at least $1$. Thus,
\[ g'(z)=\sum_{n=0}^\infty (n+1)a_{n+1}z^n \]
and
\[ g''(z)=\sum_{n=0}^\infty (n+2)(n+1)a_{n+2}z^n. \]
By inserting this into Eq.~\eqref{eq:specssHeun},
rewritten as
\begin{align*} z(z-1)(z-2)g''(z)&+\left [\tfrac72(z-1)(z-2)+\lambda z(z-2)+\tfrac12
                                  z(z-1)\right]g'(z) \\
                                &+\tfrac14\left[(\lambda+4)(\lambda+2)z-(\lambda^2+12\lambda+12)\right]g(z)=0,
\end{align*}
we obtain
\begin{align*}
  0=&(z^3-3z^2+2z)\sum_{n=0}^\infty
        (n+2)(n+1)a_{n+2}z^n \\
  &+[(\lambda+4)z^2-(2\lambda+11)z+7]
    \sum_{n=0}^\infty (n+1)a_{n+1}z^n \\
  &+\tfrac14\left[(\lambda+4)(\lambda+2)z-(\lambda^2+12\lambda+12)\right]\sum_{n=0}^\infty a_nz^n
\end{align*}
and balancing the powers of $z$, we find
\begin{align*}
  0&=\sum_{n=-1}^\infty \left
     [7(n+2)a_{n+2}-\tfrac14(\lambda^2+12\lambda+12)a_{n+1}\right]z^{n+1}
  \\
  &\quad
    +\sum_{n=0}^\infty\left[2(n+2)(n+1)a_{n+2}-(2\lambda+11)(n+1)a_{n+1}+\tfrac14(\lambda+4)(\lambda+2)a_n\right]z^{n+1}
  \\
  &\quad +\sum_{n=1}^\infty \left
    [-3(n+1)na_{n+1}+(\lambda+4)na_n\right]z^{n+1}+\sum_{n=2}^\infty
    n(n-1)a_n z^{n+1}.
\end{align*}
By setting $a_{-1}=0$, we can start all sums at $n=-1$ and
we arrive at the recurrence relation
\begin{equation}
\label{eq:recurr}
a_{n+2}=A_n(\lambda)a_{n+1}+B_n(\lambda)a_n
\end{equation}
for $n\in \{-1\}\cup \N_0$ and with
\begin{align*}
  A_n(\lambda)&:=\frac{12n^2+(8\lambda+56)n+\lambda^2+20\lambda+56}{8n^2+52n+72} \\
  B_n(\lambda)&:=-\frac{4n^2+(4\lambda+12)n+\lambda^2+6\lambda+8}{8n^2+52n+72}.
\end{align*}
In order to start the recurrence, we choose the initial
condition $a_0=1$. This freedom comes from the fact that we are
solving a linear differential equation with a one-parameter family of
solutions.

\subsection{Properties of the coefficients}

As a first and easy observation we can now rule out the existence of
polynomial solutions.

\begin{lemma}
  \label{lem:nopoly}
  Let $\Re\lambda\geq 0$ and suppose that $g: \C\to\C$ is a polynomial
  that satisfies
  Eq.~\eqref{eq:specssHeun}. Then $g=0$.
\end{lemma}

\begin{proof}
  Since $g$ is a polynomial, there exists an $N\in\N_0$ and
  coefficients $(a_n)_{n=0}^N$ such that
  \[ g(z)=\sum_{n=0}^N a_n z^n. \]
  Furthermore, by the above, the coefficients $a_n$ satisfy the
  recurrence relation Eq.~\eqref{eq:recurr}. Now observe that
  $B_n(\lambda)=0$ if and only if $\lambda\in\{-2(n+1), -2(n+2)\}$ and
  thus, $B_n(\lambda)\not=0$ for all $n\in\N_0$ and all $\lambda\in\C$
  with $\Re\lambda\geq 0$. This implies that
  \[
    a_n=-\frac{A_n(\lambda)}{B_n(\lambda)}a_{n+1}+\frac{1}{B_n(\lambda)}a_{n+2} \]
  for all $n\in\N_0$ and since $a_n=0$ for all $n>N$, we conclude that
  $a_n=0$ for all $n\in\N_0$. Consequently, $g=0$.
\end{proof}

Next, we turn to the asymptotic behavior of the coefficients. More
precisely, we are interested in the convergence radius of the series
$\sum_{n=0}^\infty a_n z^n$ and thus, we need to understand the
asymptotic behavior of the ratio $\frac{a_{n+1}}{a_n}$. To begin with,
we fix notation.

\begin{definition}
  For $\lambda\in \C$ the sequence $(a_n(\lambda))_{n\in\N_0}$ is
  defined recursively by $a_{-1}(\lambda)=0$, $a_0(\lambda)=1$, and
  \[
    a_{n+2}(\lambda)=A_n(\lambda)a_{n+1}(\lambda)+B_n(\lambda)a_n(\lambda) \]
  for $n\in \{-1\}\cup\N_0$.
\end{definition}

\begin{lemma}
  \label{lem:asym}
  Let $\Re\lambda\geq 0$.
  Then we have
  \[ \lim_{n\to\infty}\frac{a_{n+1}(\lambda)}{a_n(\lambda)}\in \{\tfrac12,1\}. \]
\end{lemma}

\begin{proof}
  We have
  \[ \lim_{n\to\infty}A_n(\lambda)=\tfrac32,\qquad
    \lim_{n\to\infty}B_n(\lambda)=-\tfrac12 \]
  and $s^2-\frac32 s+\tfrac12=0$ if and only if $s\in \{\frac12,1\}$.
  Consequently, by Poincar\'e's theorem on difference equations
  (Theorem \ref{thm:Poincare}) we either have
  \[ \lim_{n\to\infty}\frac{a_{n+1}(\lambda)}{a_n(\lambda)}\in \{\tfrac12,1\} \]
  or there exists an $N\in \N$ such that $a_n=0$ for all $n\geq N$,
  but the latter is ruled out by Lemma \ref{lem:nopoly}.
\end{proof}

If $\lim_{n\to\infty}\frac{a_{n+1}(\lambda)}{a_n(\lambda)}=\frac12$, the radius of
convergence of the series $\sum_{n=0}^\infty a_n(\lambda)z^n$ equals $2$ and in
particular,
$\sum_{n=0}^\infty a_n(\lambda)z^n$ is a solution to Eq.~\eqref{eq:specssHeun}
that belongs to $C^\infty([0,1])$. This is precisely the case we want to
rule out. Consequently, our goal is to show that
$\lim_{n\to\infty}\frac{a_{n+1}(\lambda)}{a_n(\lambda)}=1$.

Whenever $a_n(\lambda)\not=0$, we write $r_n(\lambda):=\frac{a_{n+1}(\lambda)}{a_n(\lambda)}$.
With this
notation, the recurrence relation Eq.~\eqref{eq:recurr} reads
\[
  r_{n+1}(\lambda)=\frac{a_{n+2}(\lambda)}{a_{n+1}(\lambda)}=\frac{A_n(\lambda)a_{n+1}(\lambda)+B_n(\lambda)a_n(\lambda)}{a_{n+1}(\lambda)}
  =A_n(\lambda)+\frac{B_n(\lambda)}{r_n(\lambda)} \]
for $n=0,1,2\dots$ and we keep in mind that this is only defined as long as $r_n(\lambda)\not=0$. Furthermore, we have the initial condition
\[
  r_0(\lambda)=\frac{a_1(\lambda)}{a_0(\lambda)}=a_1(\lambda)=A_{-1}(\lambda)a_0(\lambda)
  +B_{-1}(\lambda)a_{-1}(\lambda)=A_{-1}(\lambda)
  =\frac{\lambda^2+12\lambda+12}{28}. \]

\subsection{The quasi-solution}

Rephrased in terms of $r_n$, our goal is to show that
$\lim_{n\to\infty}r_n(\lambda)=1$ if $\Re\lambda\geq 0$. The idea is
now to achieve this by means of a \emph{quasi-solution}
\[ \widetilde
  r_n(\lambda):=\frac{\lambda^2}{8n^2+33n+28}+\frac{5\lambda}{5n+16}+\frac{5n+6}{5n+13},\qquad
  n\in\N, \]
which is supposed to approximate $r_n$ well enough. More precisely, we
will show that
\[ \left |\frac{r_n(\lambda)}{\widetilde r_n(\lambda)}-1\right|\leq
  \frac13 \]
for all $n\in\N$
and hence, we must have $\lim_{n\to\infty}r_n(\lambda)=1$ because by
Lemma \ref{lem:asym}, the
only other possibility is $\lim_{n\to\infty}r_n(\lambda)=\frac12$
which is not compatible with the above bound as
$\lim_{n\to\infty}\widetilde r_n(\lambda)=1$. In particular, this
estimate implies that $r_n(\lambda)\not=0$ for all $n\in\N$ and a
posteriori we see that the recursion for $r_n$ is defined for all
$n\in\N_0$. Note also that our line of reasoning provides the necessary
wiggle room for feasible estimates. Indeed, it is not
necessary to prove that
$\lim_{n\to\infty}\frac{r_n(\lambda)}{\widetilde r_n(\lambda)}=1$ directly,
which would be next to impossible.

The quasi-solution we use comes out of the blue and finding
it involves a bit of art indeed. However, there are some
rules of thumb. It is a natural first attempt to
look for a quasi-solution that is quadratic in $\lambda$ because both
$A_n(\lambda)$ and $B_n(\lambda)$ are quadratic polynomials in $\lambda$.
Then, with the help of a computer algebra system, one can look at the
first
few terms of the sequences $(r_n(0))_{n\in\N_0}$,
$(\frac12(r_n(1)-r_n(-1)))_{n\in\N_0}$, and
$(\frac12(r_n(1)-2r_n(0)+r_n(-1)))_{n\in\N_0}$ and fit simple rational
functions in $n$. Sometimes some additional tweaking is
necessary. Another approach is
based on a careful asymptotic analysis, see the corresponding
discussion in \cite{GloSch21}.

\begin{lemma}
  We have $\widetilde r_n(\lambda)\not=0$ for all $n\in\N$ and
  $\lambda\in\C$ with $\Re\lambda\geq 0$.
\end{lemma}

\begin{proof}
  Since $\widetilde r_n(\lambda)$ is completely explicit, the proof
  consists of solving a quadratic equation and is hence omitted.
\end{proof}

\begin{corollary}
  \label{cor:contholo}
  Let $n\in\N$ and $\Omega:=\{z\in \C: \Re z>0\}$.
  Then the functions
  \[ r_1: \overline\Omega\to\C,\qquad \frac{1}{\widetilde r_n}: \overline\Omega\to\C \]
  are continuous and holomorphic on $\Omega$.
\end{corollary}

\begin{definition}
  For $\lambda\in\C$ with $\Re\lambda\geq 0$ and $n\in\N_0$, we set
  \[ \delta_n(\lambda):=\frac{r_n(\lambda)}{\widetilde
      r_n(\lambda)}-1 \]
  as well as
  \[ \epsilon_n(\lambda):=\frac{A_n(\lambda)\widetilde
      r_n(\lambda)+B_n(\lambda)}{\widetilde r_n(\lambda)\widetilde
      r_{n+1}(\lambda)}-1 \]
  and
  \[ C_n(\lambda):=\frac{B_n(\lambda)}{\widetilde
      r_n(\lambda)\widetilde r_{n+1}(\lambda)}. \]
\end{definition}

Note carefully that $\epsilon_n$ and $C_n$ are explicit.

\begin{lemma}
  Let $\lambda\in \C$ with $\Re\lambda\geq 0$.
  Then the functions $\delta_n$ satisfy the recurrence relation
  \[ \delta_{n+1}(\lambda)=\epsilon_n(\lambda)-C_n(\lambda)\frac{\delta_n(\lambda)}{1+\delta_n(\lambda)} \]
  for all $n=1,2,\dots$, as long as $1+\delta_n(\lambda)\not=0$.
\end{lemma}

\begin{proof}
  This follows straightforwardly by inserting the definition of
  $\delta_n$ and by taking into account that $r_n$ satisfies the
  recurrence relation $r_{n+1}=A_n+\frac{B_n}{r_n}$.
\end{proof}

Next, we provide quantitative bounds on the functions in play.

\begin{lemma}
   \label{lem:deCbounds}
  We have the bounds
  \[ |\delta_1(it)|\leq \tfrac13,\qquad |\epsilon_n(it)|\leq
    \tfrac{1}{12},\qquad |C_n(it)|\leq \tfrac12 \]
  for all $n\in\N$ and all $t\in \R$.
\end{lemma}

\begin{proof}
  These are all bounds on explicit expressions and they can be proved
  by elementary means. For instance, we have
  \[ |C_n(it)|^2=\frac{P_n(t^2)}{Q_n(t^2)} \]
  for polynomials $P_n$ and $Q_n$. Consequently, the bound
  $|C_n(it)|\leq \frac12$ is equivalent to $Q_n(t^2)-4P_n(t^2)\geq 0$
  and the latter is trivially satisfied because the polynomial
  $Q_n(t^2)-4P_n(t^2)$ turns out to have only nonnegative coefficients
  for all $n\in\N$.
    We refer to
  \cite{CosDonGlo17} for more details.
\end{proof}

By the Phragm\'en-Lindel\"of principle, the bounds on the imaginary
axis extend to the whole complex right half-plane.

\begin{lemma}
  \label{lem:deCboundsfull}
   We have the bounds
  \[ |\delta_1(\lambda)|\leq \tfrac13,\qquad |\epsilon_n(\lambda)|\leq
    \tfrac{1}{12},\qquad |C_n(\lambda)|\leq \tfrac12 \]
  for all $n\in\N$ and all $\lambda\in\C$ with $\Re\lambda\geq 0$.
\end{lemma}

\begin{proof}
  Let $n\in\N$.
  By construction and Corollary \ref{cor:contholo}, the functions $\delta_1, \epsilon_n, C_n$ are
  continuous on the closed complex right half-plane and holomorphic on
  the open right half-plane. Furthermore, since $\delta_1$,
  $\epsilon_n$, and $C_n$ are rational functions, there exists a $K_n>0$ such
  that
  \[ |\delta_1(\lambda)|+|\epsilon_n(\lambda)|+|C_n(\lambda)|\leq K_n
    e^{|\lambda|^\frac12} \]
  for all $\lambda\in\C$ with $\Re\lambda\geq 0$. Consequently, Lemma
  \ref{lem:deCbounds} and the Phragm\'en-Lindel\"of principle (Lemma
  \ref{lem:PhrLin}) yield the claim.
\end{proof}

Now we can conclude the proof of mode stability by a simple induction.

\begin{lemma}
  We have the bound
  \[ |\delta_n(\lambda)|\leq \tfrac13 \]
  for all $n\in\N$ and all $\lambda\in\C$ with $\Re\lambda\geq 0$.
\end{lemma}

\begin{proof}
  By Lemma \ref{lem:deCboundsfull}, the claim holds for $n=1$. Assuming
  that it holds for $n$, we find, again by Lemma \ref{lem:deCboundsfull},
  \[ |\delta_{n+1}(\lambda)|\leq
    |\epsilon_n(\lambda)|+|C_n(\lambda)|\frac{|\delta_n(\lambda)|}{1-|\delta_n(\lambda)|}
    \leq \tfrac{1}{12}+\tfrac12\frac{\frac13}{\frac23}=\tfrac13 \]
   and the claim follows inductively.
 \end{proof}

 This concludes the proof of mode stability and Theorem
 \ref{thm:modstab} is established.

\section{Functional analytic setup}

In this second part, we describe the functional analytic setup for studying the
stability of the wave maps blowup. In particular, we will see how the
mode stability problem embeds into a proper operator-theoretic framework
where it occurs as the effective spectral equation for the
nonself-adjoint operator that drives the linearized evolution near the
blowup solution.

We remark that our earlier papers, e.g.~\cite{Don11, DonSchAic12}, that
implemented this approach for the first time
focused on the evolution in the backward lightcone of the
singularity. However, it is also possible to treat the problem in the
whole space with basically no additional effort. Interestingly, the
challenging spectral problems are insensitive to this
modification. Furthermore, for
the full space problem Fourier methods become available that simplify
things considerably. In particular, the treatment of the free wave
evolution in similarity coordinates can be based on the standard wave
propagators and this is the approach we present here. A systematic study of blowup stability in the whole space was
recently developed in \cite{Glo22} with a slightly different approach
that does not make explicit use of the wave propagators but
relies on abstract semigroup theory instead.

\subsection{Wave propagators}

To begin with, we recall the standard wave propagators. Our convention
for the Fourier transform is
\[ (\mc F f)(y):=\int_{\R^d}e^{-2\pi i y\cdot x}f(x)dx, \]
initially defined on the Schwartz space $\mc S(\R^d)$ and by duality extended to the
tempered distributions $\mc S'(\R^d)$.

\begin{definition}
  We define the \emph{scalar gradient} by
  \[ |\nabla|f:=\mc F^{-1}\left(2\pi|\cdot|\mc Ff\right) \]
  for $f\in \mc S(\R^d)$.
\end{definition}

\begin{remark}
  Note that $|\nabla|^2f=-\Delta f$.
\end{remark}

\begin{definition}[Wave propagators]
  For $f\in \mc S(\R^d)$ we set
  \begin{align*}
    \cos(t|\nabla|)f&:=\mc F^{-1}(\cos(2\pi t|\cdot|)\mc Ff) \\
    \frac{\sin(t|\nabla|)}{|\nabla|}f&:=\mc F^{-1}\left (\frac{\sin(2\pi
                                      t|\cdot|)}{2\pi
                                       |\cdot|}\mc F f\right).
  \end{align*}
\end{definition}
Recall that if $f,g\in \mc S(\R^d)$,
\[ u(t,\cdot):=\cos(t|\nabla|)f+\frac{\sin(t|\nabla|)}{|\nabla|}g \]
is the unique solution of the Cauchy problem
\[
  \begin{cases}
    (\partial_t^2-\Delta_x)u(t,x)=0 & \mbox{ for }(t,x)\in
    \R\times\R^d \\
    u(0,x)=f(x),\quad \partial_0 u(0,x)=g(x) & \mbox{ for }x\in \R^d
  \end{cases}.
\]
Furthermore, recall the homogeneous Sobolev norms
\[ \|f\|_{\dot
    H^s(\R^d)}:=\||\nabla|^sf\|_{L^2(\R^d)}=\|(2\pi|\cdot|)^s\mc Ff\|_{L^2(\R^d)}, \qquad
  s> -\tfrac{d}{2},\quad f\in \mc S(\R^d). \]
The wave propagators behave very well with respect to these norms.

\begin{lemma}
  \label{lem:wp}
  Let $s\geq 0$. Then we have the bounds
  \begin{align*}
    \|\cos(t|\nabla|)f\|_{\dot H^s(\R^d)}&\leq \|f\|_{\dot H^s(\R^d)}
    \\
    \left \|\frac{\sin(t|\nabla|)}{|\nabla|}f\right
    \|_{\dot H^{s+1}(\R^d)}&\leq \|f\|_{\dot H^s(\R^d)}
  \end{align*}
  for all $t\in \R$ and $f\in \mc S(\R^d)$.
\end{lemma}

\begin{proof}
  The proof is just an application of Plancherel's theorem.
\end{proof}

\subsection{The wave propagators in similarity coordinates}

Next, we switch to the similarity coordinates
\[ \tau=-\log(T-t)+\log T,\qquad \xi=\frac{x}{T-t} \]
or
\[ t=T-Te^{-\tau},\qquad x=Te^{-\tau}\xi. \]
We consider the coordinate range $\tau\geq 0$ and $\xi\in \R^d$. The
solution to the wave equation in similarity coordinates is then given
by the wave propagators in similarity coordinates.

\begin{definition}
  For $f\in \mc S(\R^d)$, $\tau\geq 0$, and $T>0$, we set
  \begin{align*}
    [C_{T}(\tau)f](\xi)&:=[\cos((T-Te^{-\tau})|\nabla|)f](Te^{-\tau}\xi) \\
    [S_{T}(\tau)f](\xi)&:=\left [\frac{\sin((T-Te^{-\tau})|\nabla|)}{|\nabla|}f\right](Te^{-\tau}\xi).
  \end{align*}
\end{definition}

The crucial observation now is the fact that the wave propagators in
similarity coordinates decay exponentially, provided one takes
sufficiently many derivatives.

\begin{lemma}
  \label{lem:decay}
  Let $s\geq 0$. Then we have the bounds
  \begin{align*}
    \|C_T(\tau)f\|_{\dot H^s(\R^d)}&\leq
                                     T^{-\frac{d}{2}+s}e^{(\frac{d}{2}-s)\tau}\|f\|_{\dot H^s(\R^d)} \\
    \|S_T(\tau)f\|_{\dot H^{s+1}(\R^d)}&\leq  T^{-\frac{d}{2}+s+1}e^{(\frac{d}{2}-s-1)\tau}\|f\|_{\dot H^s(\R^d)}
  \end{align*}
  for all $\tau\geq 0$, $T>0$, and $f\in \mc S(\R^d)$.
\end{lemma}

\begin{proof}
  This follows by a simple scaling argument combined with Lemma \ref{lem:wp}.
\end{proof}

\subsection{Back to the wave maps equation}

Now we return to the wave maps equation Eq.~\eqref{eq:wmcor},
\begin{equation}
  \label{eq:corwm2}
  \left
    (\partial_t^2-\partial_r^2-\frac{2}{r}\partial_r\right)u(t,r)+\frac{\sin(2u(t,r))}{r^2}=0.
\end{equation}
Despite its appearance, this is not a standard
radial nonlinear wave equation because of the singularity at $r=0$. A Taylor
expansion shows that smooth solutions of this equation must vanish at
$r=0$. This observation motivates the introduction of the new variable
$\widetilde v(t,r):=\frac{u(t,r)}{r}$. In terms of $\widetilde v$,
Eq.~\eqref{eq:corwm2} reads
\[ \left (\partial_t^2 -\partial_r^2 -\frac{4}{r}\partial_r
  \right)\widetilde v(t,r)+\frac{\sin(2r\widetilde
    v(t,r))-2r\widetilde v(t,r)}{r^3}=0. \]
This is now a proper radial semilinear wave equation with a smooth
nonlinearity (observe the cancellation in the numerator) but in 5
rather than 3 spatial dimensions. Thus, it is natural to formulate the
problem in terms of the function $v(t,x):=\widetilde v(t,|x|)$, where
$v(t,\cdot)$ is a radial function on $\R^5$. This leads to the
equation
\[ (\Box v)(t,x)+\frac{\sin(2|x|v(t,x))-2|x|v(t,x)}{|x|^3}=0, \]
where
\[ (\Box v)(t,x):=(\partial_t^2-\Delta_x)v(t,x) \]
denotes the \emph{d'Alembertian} or \emph{wave operator}.
In terms of the function $\widetilde
w(\tau,\xi):=v(T-Te^{-\tau},Te^{-\tau}\xi)$, we obtain
\[ T^{-2}e^{2\tau} \widetilde \Box_{\tau,\xi}\widetilde w(\tau,\xi)+
\frac{\sin(2Te^{-\tau}|\xi|\widetilde w(\tau,\xi))-2Te^{-\tau}|\xi|\widetilde
  w(t,x)}{T^3e^{-3\tau}|\xi|^3}=0, \]
where $T^{-2}e^{2\tau}\widetilde \Box_{\tau,\xi}$ is the wave operator
in similarity coordinates, i.e.,
\[ T^{-2}e^{2\tau}\widetilde \Box_{\tau,\xi}\widetilde w(\tau,\xi)=(\Box
  v)(T-Te^{-\tau},Te^{-\tau}\xi). \]
Explicitly, we have
\[
  \widetilde\Box_{\tau,\xi}=\partial_\tau^2+2\xi^j\partial_{\xi^j}\partial_\tau-(\delta^{jk}-\xi^j\xi^k)\partial_{\xi^j}\partial_{\xi^k}+\partial_\tau+2\xi^j\partial_{\xi^j}.\ \]
Note that the coefficients of
  $\widetilde\Box_{\tau,\xi}$ are independent of $\tau$.
In order to obtain an autonomous equation, we switch to the variable
$w(\tau,\xi):=Te^{-\tau}\widetilde w(\tau,\xi)$. This leads to
\begin{equation}
  \label{eq:wmssc}
  e^{-\tau}\widetilde\Box_{\tau,\xi}\left (e^\tau w(\tau,\xi)\right)
  +\frac{\sin(2|\xi|w(\tau,\xi))-2|\xi|w(\tau,\xi)}{|\xi|^3}=0.
\end{equation}
Note that
\[ e^{-\tau}\widetilde\Box_{\tau,\xi}(e^\tau w(\tau,\xi))=\left
    [\partial_\tau^2+2\xi^j\partial_{\xi^j}\partial_\tau-(\delta^{jk}-\xi^j\xi^k)\partial_{\xi^j}\partial_{\xi^k}
    +3\partial_\tau+4\xi^j\partial_{\xi^j}+2\right]w(\tau,\xi) \]
and thus, the parameter $T$ does not occur and may be formally set to
$1$. Consequently, the solution of
\[ e^{-\tau}\widetilde\Box_{\tau,\xi}(e^\tau w(\tau,\xi))=0 \]
is given by
\[ e^\tau w(\tau,\cdot)=C_1(\tau)w(0,\cdot)+S_1(\tau)[\partial_0
  w(0,\cdot)+(\cdot)^j\partial_j w(0,\cdot)+w(0,\cdot)] \]
since $(\partial_0
v)(0,\xi)=[e^\tau\partial_\tau+e^\tau\xi^j\partial_{\xi^j}](e^\tau
w(\tau,\xi))|_{\tau=0}$. Note carefully that we gain an additional
factor of decay.

Recall that we have the static solution
\[ w_*(\xi):=\frac{2}{|\xi|}\arctan(|\xi|) \]
which we want to perturb. Thus, we plug in the ansatz
$w(\tau,\xi)=w_*(\xi)+\varphi(\tau,\xi)$
and obtain the equation
\begin{equation}
  \label{eq:wmsscpert}
  e^{-\tau}\widetilde\Box_{\tau,\xi}\left (e^\tau
    \varphi(\tau,\xi)\right)
  +\frac{2|\xi|\cos(2|\xi|w_*(\xi))\varphi(\tau,\xi)-2|\xi|\varphi(\tau,\xi)}{|\xi|^3}+N(\varphi(\tau,\xi),\xi)=0,
\end{equation}
where
\[
  N(y,\xi):=\frac{\sin(2|\xi|(w_*(\xi)+y))-\sin(2|\xi|w_*(\xi))-2|\xi|\cos(2|\xi|w_*(\xi))y}{|\xi|^3}. \]
Note that $N(y,\xi)$ is quadratic in $y$ and smooth in $\xi$.

\subsection{Semigroup formulation}

By introducing the variable
\[ \Phi(\tau)(\xi)=
  \begin{pmatrix}
    \varphi(\tau,\xi) \\
    (\partial_\tau+\xi^j\partial_{\xi^j}+1)\varphi(\tau,\xi)
  \end{pmatrix}, \]
Eq.~\eqref{eq:wmsscpert} can be written as the first-order system
\begin{equation}
  \label{eq:wmsscpertSG}
  \partial_\tau \Phi(\tau)=\widehat{\mb L}_0\Phi(\tau)+\mb L'\Phi(\tau)+\mb
  N(\Phi(\tau)),
\end{equation}
with the \emph{formal} differential operator
\[ \widehat{\mb L}_0\begin{pmatrix}f_1\\f_2\end{pmatrix}:=
  \begin{pmatrix}
    -\Lambda f_1-f_1+f_2 \\
    \Delta f_1-\Lambda f_2-2f_2
  \end{pmatrix},\qquad (\Lambda f)(\xi):=\xi^j\partial_{\xi^j}f(\xi),
\]
and
\[ \left[\mb L'
  \begin{pmatrix}
    f_1 \\ f_2
  \end{pmatrix}\right](\xi):=
\begin{pmatrix}0 \\ -\frac{2\cos(2|\xi|w_*(\xi))-2}{|\xi|^2}f_1(\xi)
\end{pmatrix}
=\begin{pmatrix}0 \\ \frac{16}{(1+|\xi|^2)^2}f_1(\xi)
\end{pmatrix},
\]
as well as
\[ \mb N\left (
    \begin{pmatrix}
      f_1 \\ f_2
    \end{pmatrix}
    \right )(\xi):=
    \begin{pmatrix}
      0 \\ -N(f_1(\xi),\xi)
    \end{pmatrix}.
  \]
  By construction, the solution of the Cauchy problem
  $\partial_\tau\Phi(\tau)=\widehat{\mb L}_0\Phi(\tau)$, $\Phi(0)=\mb
  f=(f_1,f_2)\in \mc S(\R^d)\times \mc S(\R^d)$, is given by
  \begin{align*} \Phi(\tau)(\xi)&=
    \begin{pmatrix}
      e^{-\tau}[C_1(\tau)f_1](\xi)+e^{-\tau}[S_1(\tau)f_2](\xi) \\
      (\partial_\tau+\xi^j\partial_{\xi^j}+1)[e^{-\tau}C_1(\tau)f_1](\xi)+(\partial_\tau+\xi^j\partial_{\xi^j}+1)[e^{-\tau}S_1(\tau)f_2](\xi)
    \end{pmatrix} \\
    &=:[\mb S_0(\tau)\mb f](\xi).
  \end{align*}
  Note that
  \begin{align*}
    \partial_\tau [C_1(\tau)f](\xi)
    &=\partial_\tau
      \left[\cos((1-e^{-\tau})|\nabla|)f\right](e^{-\tau}\xi)
    \\
    &=-e^{-\tau}\left[\sin((1-e^{-\tau})|\nabla|)|\nabla|f\right](e^{-\tau}\xi)
      -e^{-\tau}\xi^j\partial_j\left[\cos((1-e^{-\tau})|\nabla|)f\right](e^{-\tau}\xi)
    \\
    &=e^{-\tau}\left [S_1(\tau)\Delta f\right](\xi)
      -\xi^j\partial_{\xi^j}\left[C_1(\tau)f\right](\xi)
  \end{align*}
  and thus,
  \begin{align*}
    (\partial_\tau+\xi^j\partial_{\xi^j}+1)\left
    [e^{-\tau}C_1(\tau)f\right ](\xi)=e^{-2\tau}\left
    [S_1(\tau)\Delta f\right](\xi).
  \end{align*}
  Analogously,
  \[
    (\partial_\tau+\xi^j\partial_{\xi^j}+1)\left
      [e^{-\tau}S_1(\tau)f\right ](\xi)=e^{-2\tau}[C_1(\tau)f](\xi) \]
  and this yields the representation
  \[ \mb S_0(\tau)\mb f=
    \begin{pmatrix}
      e^{-\tau}C_1(\tau)f_1+e^{-\tau}S_1(\tau)f_2 \\
      e^{-2\tau}S_1(\tau)\Delta f_1+ e^{-2\tau}C_1(\tau)f_2
    \end{pmatrix}. \]
  Consequently, by Lemma \ref{lem:decay}, we obtain the bound
  \[ \|\mb S_0(\tau)\mb f\|_{\dot H^s(\R^d)\times \dot H^{s-1}(\R^d)}\lesssim
   e^{(\frac{d}{2}-1-s)\tau}\|\mb f\|_{\dot H^s(\R^d)\times \dot
     H^{s-1}(\R^d)} \]
 for all $\mb f\in \mc S(\R^d)\times \mc S(\R^d)$, $\tau\geq 0$, and any $s\geq 0$.
 In particular, $\mb S_0$ extends to a semigroup on
 $\mc H:=(\dot H^2(\R^5)\times \dot H^1(\R^5))\cap (\dot H^4(\R^5)\times
 \dot H^3(\R^5))$ with operator norm satisfying
 \[ \|\mb S_0(\tau)\|_{\mc H}\lesssim e^{-\frac12\tau} \]
 for all $\tau\geq 0$.
 Furthermore, it is a simple exercise to show that the map
 $\tau\mapsto \mb S_0(\tau)\mb f: [0,\infty)\to \mc H$ is continuous
 for any $\mb f\in\mc H$
 and the whole abstract machinery of semigroup theory applies.
 In view of the nonlinear problem, it is also crucial that we have an
 $L^\infty$-embedding of \emph{intersection Sobolev spaces} which implies that $\mc H$ is a Banach algebra:

 \begin{lemma}
   \label{lem:alg}
  Let $0\leq s<\frac{d}{2}<t$ and $H:=\dot H^s(\R^d)\cap \dot
  H^t(\R^d)$. Then we have the Sobolev embedding $H\hookrightarrow C(\R^d)\cap
  L^\infty(\R^d)$ and
  \[ \|fg\|_H\lesssim \|f\|_H\|g\|_H \]
  for all $f,g\in \mc S(\R^d)$.
\end{lemma}

\begin{proof}
  Left as an exercise.
\end{proof}

\subsection{Spectral analysis of the generator}

By construction,
 \[ \partial_\tau \mb S_0(\tau)\mb f=\widehat{\mb L}_0\mb S_0(\tau)\mb f \]
 for all $\mb f\in \mc S(\R^5)\times \mc S(\R^5)=:\mc D(\widehat{\mb L}_0)$ and thus,
the generator of $\mb S_0$ is the closure $\mb L_0: \mc D(\mb
L_0)\subset \mc H\to\mc H$ of the operator $\widehat{\mb L}_0: \mc D(\widehat{\mb
L}_0)\subset \mc H\to\mc H$. Since $\mb L_0$ is an abstract object, we
need the following auxiliary result in order to get our hands on the
spectral problem for $\mb L_0$ and its perturbations.

\begin{lemma}
  \label{lem:dist}
  Let $\mb f\in \mc D(\mb L_0)$ and set $\mb g:=(g_1,g_2):=\mb L_0 \mb f$. Then
  $\mb f=(f_1,f_2)$ satisfies
  \begin{align*}
    g_1(\xi)&=-\xi^j\partial_{\xi^j}f_1(\xi)-
              f_1(\xi)+f_2(\xi) \\
    g_2(\xi)&=\Delta f_1(\xi)-\xi^j\partial_{\xi^j}f_2(\xi)-2f_2(\xi)
  \end{align*}
  in the sense of distributions.
\end{lemma}

\begin{proof}
  Let $\mb f\in \mc D(\mb L_0)$. By definition of the closure, there
  exists a sequence $(\mb f_n)_{n\in\N}\subset \mc D(\widehat{\mb L}_0)=\mc
  S(\R^5)\times \mc S(\R^5)$ such that $\lim_{n\to\infty}\|\mb f_n-\mb
  f\|_{\mc H}=0$ and $\lim_{n\to\infty}\|\mb g_n-\mb g\|_{\mc H}=0$,
  where $\mb g_n:=\widehat{\mb L}_0 \mb f_n$. By definition,
  \begin{align*}
    g_{1n}(\xi)&=-\xi^j\partial_{\xi^j}f_{1n}(\xi)-
              f_{1n}(\xi)+f_{2n}(\xi) \\
    g_{2n}(\xi)&=\Delta f_{1n}(\xi)-\xi^j\partial_{\xi^j}f_{2n}(\xi)-2f_{2n}(\xi),
  \end{align*}
  where $\mb f_n=(f_{1n}, f_{2n})$ and $\mb
  g_n=(g_{1n},g_{2n})$. Consequently, by testing these equations with
  a function in $C^\infty_c(\R^5)$, integrating by parts, and taking
  the limit $n\to\infty$, the claim follows thanks to the Sobolev
  embedding $\mc H\hookrightarrow L^\infty(\R^5)\times L^\infty(\R^5)$.
\end{proof}

Next, we turn to the full linear operator $\mb L:=\mb L_0+\mb L'$ with
$\mc D(\mb L):=\mc D(\mb L_0)$ as it
occurs in Eq.~\eqref{eq:wmsscpertSG}. Here, we have a nice compactness property.

\begin{lemma}
  \label{lem:compact}
  The operator $\mb L': \mc H\to \mc H$ is compact.
\end{lemma}

\begin{proof}
  By definition, we have
  \[ \mb L'
    \begin{pmatrix}
      f_1 \\ f_2
    \end{pmatrix}
    =
    \begin{pmatrix}
      0 \\ Vf_1
    \end{pmatrix},\qquad V(\xi):=\frac{16}{(1+|\xi|^2)^2}.
  \]
  Let $(\mb f_n)_{n\in\N}$ be a bounded sequence
  in $\mc H$, where we write $\mb f_n=(f_{1n},f_{2n})$.
  Furthermore, let $\chi: \R^5\to [0,1]$ be a smooth cut-off that satisfies
  $\chi(\xi)=1$ if $|\xi|\leq 1$ and $\chi(\xi)=0$ if $|\xi|\geq 2$.
  By Lemma \ref{lem:alg}, we have
  \begin{align*} \|\mb L'\mb f_n\|_{\mc H}&=\|Vf_{1n}\|_{\dot H^3(\R^5)\cap \dot
      H^1(\R^5)}\lesssim
    \|\chi_kVf_{1n}\|_{H^3(\R^5)}+\|(1-\chi_k)Vf_{1n}\|_{\dot H^1(\R^5)\cap
                                            \dot H^3(\R^5)} \\
    &\lesssim \|f_{1n}\|_{H^3(\B^5_{2k})}+\|(1-\chi_k)Vf_{1n}\|_{\dot
      H^1(\R^5)}
      +\|(1-\chi_k)Vf_{1n}\|_{\dot H^2(\R^5)\cap
                                            \dot H^4(\R^5)}
    \end{align*}
for all $n,k\in\N$, where $\chi_k(\xi):=\chi(\frac{\xi}{k})$. Now we
employ Hardy's inequality (see e.g.~\cite{MusSch13}, p.~243, Theorem
  9.5) and the decay of $V$ to obtain the bound
\begin{align*}
  \|(1-\chi_k)Vf_{1n}\|_{\dot H^1(\R^5)}
  &\lesssim
  \||\nabla|[(1-\chi_k)V]f_{1n}\|_{L^2(\R^5)}
    +\|(1-\chi_k)V|\nabla|f_{1n}\|_{L^2(\R^5)} \\
  &\lesssim
    k^{-1}\||\cdot|^{-2}f_{1n}\|_{L^2(\R^5)}+k^{-1}\||\cdot|^{-1}|\nabla|f_{1n}\|_{L^2(\R^5)}
  \\
  &\lesssim k^{-1}\|f_{1n}\|_{\dot H^2(\R^5)}.
\end{align*}
Combined with the above, this leads to the estimate
\[ \|\mb L'\mb f_n\|_{\mc H}\lesssim
  \|\mb f_n\|_{H^3(\B_{2k}^5)\times H^2(\B_{2k}^5)}+k^{-1}\|\mb
  f_n\|_{\mc H}\lesssim \|\mb f_n\|_{H^3(\B_{2k}^5)\times H^2(\B_{2k}^5)}+k^{-1} \]
for all $n,k\in\N$.
By the Sobolev embedding $\mc H\hookrightarrow L^\infty(\R^5)\times
L^\infty(\R^5)$
and H\"older's inequality, we
obtain
\[ \mc H\hookrightarrow H^4_{\mathrm{loc}}(\R^5)\times
  H^3_{\mathrm{loc}}(\R^5) \hookrightarrow H^3_{\mathrm{loc}}(\R^5)\times
  H^2_{\mathrm{loc}}(\R^5) \]
and by the compactness of the latter embedding, there exists, for each
$k\in \N$, a subsequence $(\mb f_{k,n})_{n\in\N}$ of $(\mb
f_n)_{n\in\N}$ that converges in $H^3(\B^5_{2k})\times H^2(\B^5_{2k})$
and such that $(\mb f_{k+1,n})_{n\in\N}$ is a subsequence of $(\mb
f_{k,n})_{n\in\N}$.
In particular, there exists an $N(k)\in \N$ such that
\[ \|\mb f_{k,m}-\mb f_{k,n}\|_{H^3(\B^5_{2k})\times
    H^2(\B^5_{2k})}\leq k^{-1} \]
for all $m,n\geq N(k)$. Clearly, we may choose $k\mapsto N(k):
\N\to\N$ to be monotonically increasing.
We set $\mb g_k:=\mb f_{k,N(k)}$ for $k\in\N$.
Then $(\mb g_k)_{k\in\N}$ is a subsequence of $(\mb f_n)_{n\in\N}$ and
we have
\begin{align*}
  \|\mb L'\mb g_{k+\ell}-\mb L'\mb g_k\|_{\mc H}&\lesssim \|\mb
  g_{k+\ell}-\mb g_k\|_{H^3(\B^5_{2k})\times H^2(\B^5_{2k})}+k^{-1} \\
  &=\|\mb
  f_{k+\ell, N(k+\ell)}-\mb f_{k,N(k)}\|_{H^3(\B^5_{2k})\times
    H^2(\B^5_{2k})}+k^{-1} \\
  &\lesssim k^{-1}
\end{align*}
for all $k,\ell\in\N$ because $(\mb f_{k+\ell,n})_{n\in\N}$ is a
subsequence of $(\mb f_{k,n})_{n\in\N}$. Consequently, $(\mb L'\mb
g_k)_{k\in\N}$ is a Cauchy sequence in $\mc H$ and this proves the claim.
\end{proof}

As a consequence of the bound $\|\mb S_0(\tau)\|_{\mc H}\lesssim
e^{-\frac12 \tau}$, we see that the \emph{free resolvent}
\[ \mb R_{\mb L_0}(\lambda):=(\lambda\mb I-\mb L_0)^{-1}=\int_0^\infty
  e^{-\lambda\tau}\mb S_0(\tau)d\tau \]
exists provided that $\Re\lambda>-\frac12$.
By the \emph{Birman-Schwinger principle}, i.e., the identity
\[ \lambda\mb I-\mb L=[\mb I-\mb L'\mb R_{\mb
    L_0}(\lambda)](\lambda\mb I-\mb L_0), \]
it follows that $\lambda\mb I-\mb L$ is bounded invertible for
$\Re\lambda>-\frac12$ if and only if $\mb I-\mb L'\mb R_{\mb
  L_0}(\lambda)$ is bounded invertible.
This observation leads to the important result that possible spectral
points $\lambda$ of $\mb L$ with $\Re\lambda\geq-\frac14$, say, are confined to
a compact region. In order to formulate the exact statement, we define
\[ \Gamma_R:=\{z\in\C: \Re z> -\tfrac14, |z|> R\} \]
for $R>0$.

\begin{lemma}
\label{lem:speccpt}
  There exists an $R>0$ such that
  $\sigma(\mb L)\cap \Gamma_R=\emptyset$.
  Furthermore, we have
  \[ \sup_{\lambda\in\Gamma_R}\|(\lambda\mb I-\mb L)^{-1}\|_{\mc
      H}<\infty. \]
\end{lemma}

\begin{proof}
Let $\Re\lambda\geq -\frac14$, $\mb g\in \mc H$, and set $\mb f:=\mb
R_{\mb L_0}(\lambda)\mb g$. Then $\mb f\in \mc D(\mb L_0)$ and
$(\lambda \mb I-\mb L_0)\mb f=\mb g$. By Lemma \ref{lem:dist}, we have
\[ g_1(\xi)=\xi^j\partial_{\xi^j}f_1(\xi)+(\lambda+1)f_1(\xi)-f_2(\xi) \]
in the sense of distributions, where $\mb g=(g_1,g_2)$ and $\mb
f=(f_1,f_2)$. Consequently,
\[ f_1(\xi)=\frac{1}{\lambda+1}\left
    [-\xi^j\partial_{\xi^j}f_1(\xi)+f_2(\xi)+g_1(\xi)\right] \]
and Hardy's inequality yields
\begin{align*}
  \|\mb L'\mb R_{\mb L_0}(\lambda)\mb g\|_{\mc H}
  &=\|Vf_1\|_{\dot
    H^1(\R^5)\cap \dot H^3(\R^5)}\lesssim \frac{1}{|\lambda+1|}\left (\|\mb
    f\|_{\mc H}+\|\mb
    g\|_{\mc H}\right) \\
    &=\frac{1}{|\lambda+1|}\left (\|\mb
    R_{\mb L_0}(\lambda)\mb g\|_{\mc H}+\|\mb
    g\|_{\mc H}\right) \\
  &\lesssim \frac{1}{|\lambda+1|}\|\mb g\|_{\mc H}
\end{align*}
for all $\lambda\in \C$ with $\Re\lambda\geq -\frac14$, where we have exploited
the decay of the potential $V(\xi)=\frac{16}{(1+|\xi|^2)^2}$ as in the
proof of Lemma \ref{lem:compact}. Thus, if $R>0$ is chosen large
enough, we obtain $\|\mb L'\mb R_{\mb L_0}(\lambda)\|_{\mc H}\leq \frac12$
for all $\lambda\in \Gamma_R$ 
and the existence of $[\mb I-\mb L'\mb R_{\mb L_0}(\lambda)]^{-1}$
follows by a Neumann series argument. By the Birman-Schwinger
principle, this implies the claim.
\end{proof}

\subsection{Connection to mode stability}
Since $\mb L'$ is compact, it follows
from Lemma \ref{lem:speccpt} and the \emph{analytic Fredholm theorem}
(see e.g.~\cite{Sim15}, p.~194, Theorem 3.14.3) that
$\lambda\mb I-\mb L$ is bounded invertible for all $\lambda\in \C$
with $\Re\lambda\geq-\frac14$ except for a finite number of eigenvalues,
each with finite algebraic multiplicity.

In order to locate these eigenvalues, we need to solve the equation
$(\lambda\mb I-\mb L)\mb f=\mb 0$.
Suppose there exists a (nontrivial) solution $\mb f=(f_1,f_2)\in \mc D(\mb L)=\mc D(\mb L_0)$. By Lemma
\ref{lem:dist}, we see that
\begin{align*}
  \xi^j\partial_{\xi^j}f_1(\xi)+(\lambda+1)f_1(\xi)-f_2(\xi)
  &=0 \\
  -\Delta
  f_1(\xi)+\xi^j\partial_{\xi^j}f_2(\xi)+(\lambda+2)f_2(\xi)-\frac{16}{(1+|\xi|^2)^2}f_1(\xi)
  &=0
\end{align*}
in the sense of distributions and by inserting the first equation into
the second one, we find
\begin{equation}
  \label{eq:specxi}
  -(\delta^{jk}-\xi^j\xi^k)\partial_{\xi^j}\partial_{\xi^k}f_1(\xi)+2(\lambda+2)\xi^j\partial_{\xi^j}f_1(\xi)
  +(\lambda+1)(\lambda+2)f_1(\xi)-\frac{16}{(1+|\xi|^2)^2}f_1(\xi)=0.
  \end{equation}
Furthermore, by Sobolev embedding and H\"older's inequality, we have $\dot H^2(\R^5)\cap \dot
H^4(\R^5)\subset L^\infty(\R^5) \subset L^2_{\mathrm{loc}}(\R^5)$ and
thus, $f_1\in H^4_{\mathrm{loc}}(\R^5)$.
  Consequently, by elliptic regularity, we
conclude that $f_1\in C^\infty(\R^5\setminus\S^4)$
and $f_1$ satisfies Eq.~\eqref{eq:specxi}  on $\R^5\setminus\S^4$ in the
sense of classical derivatives.
Recall that $f_1$ is radial and thus,
in terms of the auxiliary function $\widehat f_1\in
C^\infty(\R\setminus\{-1, 1\})$,
given by $\widehat f_1(\rho):=\rho f_1(\rho e_1)$,
Eq.~\eqref{eq:specxi} reads
\[ -(1-\rho^2)\widehat f_1''(\rho)-\frac{2}{\rho}\widehat f_1'(\rho)+2(\lambda+1)\rho
  \widehat
  f_1'(\rho)+\lambda(\lambda+1)\widehat f_1(\rho)+2\frac{1-6\rho^2+\rho^4}{\rho^2(1+\rho^2)^2}\widehat
  f_1(\rho)=0. \]
Consequently, since $\widehat f_1\in C^\infty(\R\setminus\{-1,1\})\cap
H^4((\frac12,\frac32))$, it follows by Frobenius' method that
$\widehat f_1\in C^\infty(\R)$ and $\widehat f_1$ is a nontrivial
solution in $C^\infty([0,1])$ of Eq.~\eqref{eq:spec}.
 This is the connection to the mode stability
problem that gives the latter a proper functional analytic
interpretation.

\begin{proposition}
  \label{prop:spec}
  For the spectrum $\sigma(\mb L)$ of $\mb L$ we have
  \[ \sigma(\mb L)\subset \{z\in\C: \Re z<0\}\cup \{1\} \]
  and $1$ is an eigenvalue of $\mb L$.
\end{proposition}

\begin{proof}
  The statement about the spectrum follows from the above derivation
  in conjunction with Lemma
  \ref{lem:speccpt} and mode stability
  (Theorem \ref{thm:modstab}).
  The eigenvalue $1$ stems from time translation symmetry as explained
  in the discussion of mode stability. 
\end{proof}

\subsection{Control of the linearized evolution}

We define the Riesz or spectral projection associated to the
eigenvalue $1$ by
\[ \mb P:=\frac{1}{2\pi i}\int_\gamma (z\mb I-\mb L)^{-1}dz, \]
where $\gamma: [0,1]\to\C$ is given by $\gamma(t):=1+\tfrac12 e^{2\pi
  i t}$.

\begin{lemma}
  The eigenvalue $1$ of $\mb L$ is simple, i.e., $\mb P$ has rank $1$.
\end{lemma}

\begin{proof}
  By the analytic Fredholm theorem (see above), $\mb P$ has finite
  rank. Consequently,
  $\mb I-\mb L$ restricts to a finite-dimensional operator on $\rg\mb
  P$ and from linear algebra we infer that the part of $\mb I-\mb L$
  in $\rg\mb P$ is nilpotent.
  The fact that the eigenvalue $1$ has algebraic multiplicity
  exactly equal to $1$ is then proved by ODE methods, by showing that the
  equation $(\mb I-\mb L)\mb f=\mb f_*$, where $\mb f_*$ is the eigenfunction
  associated to the eigenvalue $1$, has no solution. This is an
  exercise with the variation of parameters formula that we leave to
  the interested reader, see e.g.~\cite{DonSchAic12}, Lemma 4.20, for
  the precise argument.
\end{proof}

From abstract semigroup theory we can now obtain a sufficiently
detailed understanding of the linearized evolution generated by $\mb
L$.

\begin{lemma}
  The operator $\mb L$ generates a strongly-continuous semigroup $\mb
  S$ on $\mc H$. Furthermore,
  there exists an $\epsilon>0$ and a $C>0$ such that
  \begin{align*}
    \|\mb S(\tau)(\mb I-\mb P)\mb f\|_{\mc H}&\leq C
                                               e^{-\epsilon\tau}\|(\mb
                                               I-\mb P)\mb f\|_{\mc H}
    \\
    \mb S(\tau)\mb P\mb f&=e^{\tau}\mb P\mb f
  \end{align*}
  for all $\tau\geq 0$ and $\mb f\in \mc H$.
\end{lemma}

\begin{proof}
  The operator $\mb L$ differs from the semigroup
  generator $\mb L_0$ by the bounded operator $\mb L'$ and hence
  generates a semigroup $\mb S$ itself.
The statement about the evolution on the unstable subspace, $\mb
S(\tau)\mb P\mb f=e^\tau\mb P\mb f$, is a direct consequence of the
fact that the range of $\mb P$ is one-dimensional and hence spanned by the eigenfunction of $\mb
L$ associated to the eigenvalue $1$. For the evolution on the stable
subspace, we note that Lemma \ref{lem:speccpt} and Proposition
\ref{prop:spec} imply that 
\[ \sup_{\Re\lambda>0}\|(\lambda \mb I-\mb L)^{-1}(\mb I-\mb
  P)\|_{\mc H}<\infty \]
and the claimed growth bound follows from the Gearhart-Pr\"u\ss-Greiner-Theorem,
see e.g.~\cite{EngNag00}, p.~302, Theorem 1.11.
\end{proof}

\subsection{The nonlinear problem}

We finally sketch how to proceed with the nonlinear stability. 
In Duhamel form, the equation we would like to solve reads
\[ \Phi(\tau)=\mb S(\tau)\mb f+\int_0^\tau \mb S(\tau-\tau')\mb
  N(\Phi(\tau'))d\tau'.\]
Typically, such an equation is solved by a fixed point
argument. However, in the present form this is not possible due to
the exponential growth of the semigroup on $\rg\mb P$. Thus, we
borrow an idea from dynamical systems theory known as the
\emph{Lyapunov-Perron method} and consider instead the equation
\begin{equation}
  \label{eq:mod}
  \Phi(\tau)=\mb S(\tau)[\mb f-\mb C(\mb f, \Phi)]+\int_0^\tau \mb
  S(\tau-\tau')\mb N(\Phi(\tau'))d\tau',
  \end{equation}
where
\[ \mb C(\mb f, \Phi):=\mb P\mb f+\mb P\int_0^\infty e^{-\tau'}\mb
  N(\Phi(\tau'))d\tau' \]
is a correction term that stabilizes the evolution. Formally, this
term is obtained by applying the spectral projection $\mb P$ to the
original equation.
Consequently, the subtraction of $\mb C(\mb f, \Phi)$ corrects the
initial data along the one-dimensional subspace $\rg\mb P$ on which
the linearized evolution grows exponentially. Note, however, that
there is a nonlinear self-interaction, i.e., the correction term
depends on the solution itself and is not known in advance as would be
the case for a linear problem. Nonetheless, by a routine fixed point
argument utilizing the Banach algebra property of $\mc H$, we can
show that Eq.~\eqref{eq:mod} has a solution $\Phi\in C([0,\infty),\mc
H)$ for any small data $\mb f$. Finally, by realizing that the data we
want to describe depend on $T$, we see, e.g.~by the intermediate value
theorem,
that there always exists a $T$
that makes the correction term vanish. By translating back to the
original variables, we finally arrive at the following result on
the stability of the wave maps blowup. Recall the
wave maps equation in corotational symmetry reduction,
\begin{equation}
  \label{eq:wmcor5d}
  (\Box v)(t,x)+\frac{\sin(2|x|v(t,x))-2|x|v(t,x)}{|x|^3}=0,
\end{equation}
and the self-similar blowup solution
\[ v_*^T(t,x):=(T-t)^{-1}w_*\left(\frac{x}{T-t}\right),\quad
  w_*(\xi)=\frac{2}{|\xi|}\arctan(|\xi|). \]

\begin{theorem}[Nonlinear asymptotic stability of wave maps blowup]
  There exist constants $M, \delta_0>0$ such that the following holds. Let $\delta\in [0,\delta_0]$ and suppose that $f,g\in C^\infty(\R^5)$
  are radial and satisfy
  \[ \|f-v_*^1(0,\cdot)\|_{\dot H^2(\R^5)\cap \dot
      H^4(\R^5)}+\|g-\partial_0 v_*^1(0,\cdot)\|_{\dot H^1(\R^5)\cap
      \dot H^3(\R^5)}\leq \frac{\delta}{M}. \]
  Then there exist a $T\in [1-\delta, 1+\delta]$ and a unique
  solution $v\in C^\infty([0,T)\times \R^5)$ of
  Eq.~\eqref{eq:wmcor5d} that satisfies $(v(0,\cdot),\partial_0
  v(0,\cdot))=(f,g)$.
  Furthermore, we have the decomposition
  \[ v(t,x)=(T-t)^{-1}\left [w_*\left(\frac{x}{T-t}\right)+\varepsilon\left(t,\frac{x}{T-t}\right)\right] \]
  where
  \[ \|\varepsilon(t,\cdot)\|_{\dot H^2(\R^5)\cap\dot H^4(\R^5)}
    +\left \|(T-t)\partial_t \varepsilon(t,\cdot)+\Lambda
      \varepsilon(t,\cdot)+\varepsilon(t,\cdot)\right\|
    _{\dot H^1(\R^5)\cap \dot H^3(\R^5)}\to
    0 \]
  as $t\to T-$.
\end{theorem}

\begin{remark}
  Analogous results are known in all supercritical
  dimensions, see \cite{ChaDonGlo17, Glo22}. The stability problem
  outside of corotational symmetry is still open, though. However,
  there are stability results on self-similar blowup without symmetry
  assumptions for the simpler wave equation with a power nonlinearity
  \cite{DonSch16, GloSch21, CsoGloSch21, Ost23}.
\end{remark}

\section{Conclusion}

Understanding large-data solutions of supercritical evolution equations remains
one of the great challenges in contemporary analysis. The only
rigorous methods we have depend on the existence of special solutions that
are sufficiently well known. 
In many cases, self-similar solutions play this role and provide an
entrance point to the rigorous study of large-data regimes because
they open up the possibility of perturbative treatments. However, the
understanding of the linearized evolution close to a self-similar
solution is very challenging and requires knowledge of the
spectrum of the corresponding linear operator that is genuinely
nonself-adjoint in case of wave equations. This is the point where the
analysis typically fails because there are no general methods to treat
these spectral problems. In this exposition we presented the only
known method so far that is capable of extracting the necessary
spectral information in a number of nontrivial cases in a rigorous
way. It consists of a ``hard part'' that proves the mode stability and
a ``soft part'' that embeds the mode stability problem into a
proper spectral-theoretic framework for the generator of the
linearized evolution. Once the linearized evolution is understood, the
treatment of the full nonlinear problem is routine.
The approach was successfully applied to some of the most
important models such as wave maps \cite{Don11, DonSchAic12,
  CosDonGlo17, ChaDonGlo17, DonGlo19, Glo22}, Yang-Mills fields
\cite{Don14, CosDonGloHua16, Glo22a, Glo23}, and wave
equations with power nonlinearities \cite{DonSch14, DonSch16,
  DonSch17, GloSch21, CsoGloSch21}. In addition, extensions to more
general coordinate systems \cite{BieDonSch21, DonOst21,
  CheDonGloMcNSch22, Ost23} and weaker topologies \cite{Don17,
  DonRao20, Wal22, DonWal22a, DonWal22b} were considered.

Despite this recent success, a lot remains to be done. The presented method relies
strongly on the fine properties of the perturbed solution and is
probably hard to implement if the solution is not known in closed
form. Consequently, it would be
very desirable to develop more conceptual methods that provide
easy-to-verify criteria for mode stability. Whether this is possible
at all remains to be seen. A conceptual breakthrough in this area
would constitute a major step forward in modern PDE analysis and open
a whole new spectrum of problems that could be rigorously dealt with. 
In any case, we hope that this exposition
provides an accessible account to the current method that may be useful for researchers that
are confronted with similar problems in their work. 

\appendix

\section{Background material}

\subsection{The Phragm\'en-Lindel\"of principle}

The Phragm\'en-Lindel\"of principle is an extension of the maximum
principle to unbounded domains. There are many different versions and
we present the simplest one that is sufficient for our purposes. First,
recall the fundamental maximum principle from complex analysis.

\begin{lemma}[Maximum principle]
  Let $\Omega\subset \C$ be open, connected, and bounded. Suppose that $f:
  \overline{\Omega}\to \C$ is continuous and that $f|_{\Omega}:
  \Omega\to\C$ is holomorphic. Then
  \[ |f(z)|\leq \max_{\zeta\in \partial\Omega}|f(\zeta)| \]
  for all $z\in \Omega$.
\end{lemma}

The maximum principle shows that if we want to control a holomorphic
function on a bounded domain $\Omega$, it is enough to control it on
the boundary.
The assumption of boundedness is crucial here. However, under a mild
growth condition, the maximum principle extends to unbounded
domains and in this situation it goes by the name of
Phragm\'en-Lindel\"of. There are many different versions of this
principle. We use a very basic one that allows us to bound a function
on the complex right half-plane by its values on the imaginary axis.

\begin{lemma}[Phragm\'en-Lindel\"of principle]
  \label{lem:PhrLin}
  Let $\Omega:=\{z\in \C: \Re z>0\}$ and suppose that $f:
  \overline{\Omega}\to \C$ is continuous and that $f|_{\Omega}:
  \Omega\to\C$ is holomorphic. Let $M>0$. If
\begin{enumerate}
  \item $|f(it)|\leq M$ for
    all $t\in \R$ and
    \item there exists a $C\geq 0$ such that $|f(z)|\leq
      Ce^{|z|^\frac12}$ for all $z\in \Omega$
    \end{enumerate}
    then
  \[ |f(z)|\leq M \]
  for all $z\in \Omega$.
\end{lemma}

\begin{proof}
The proof is very simple and plays the situation back to the standard
maximum principle.
First, we note that the function $z\mapsto z^\frac34: \overline\Omega\to \C$
is continuous and holomorphic on $\Omega$.
Then, for $\epsilon>0$, we define an auxiliary function
$f_\epsilon: \overline{\Omega}\to\C$ by $f_\epsilon(z):=e^{-\epsilon
  z^\frac34}f(z)$. Again, $f_\epsilon$ is continuous and holomorphic
on $\Omega$. Furthermore, 
\[ |f_\epsilon(z)|=e^{-\epsilon\Re
    z^\frac34}|f(z)|=e^{-\epsilon|z|^\frac34\cos(\frac34\arg
    z)}|f(z)| \]
for all $z\in \overline\Omega$ and thus,
\[ |f_\epsilon(it)|=e^{-\epsilon
    |t|^\frac34\cos(\frac34\frac{\pi}{2})}|f(it)|\leq |f(it)|\leq M \]
for all $t\in\R$ and $\epsilon>0$ because
$\eta:=\cos(\frac34\frac{\pi}{2})>0$.
Next, we have the bound
\[ |f_\epsilon(z)|\leq e^{-\epsilon\eta|z|^\frac34}|f(z)|\leq
  Ce^{-\epsilon\eta
    |z|^\frac34+|z|^\frac12}=Ce^{-|z|^\frac34(\epsilon\eta-|z|^{-\frac14})}\to
  0 \]
as $|z|\to\infty$ and thus, $|f_\epsilon(z)|\leq M$ if $|z|$ is
sufficiently large. For $R>0$ we define the domain
\[ \Omega_R:=\{z\in \C: |z|<R\}\cap \Omega. \]
By the above, $f_\epsilon$ is holomorphic on $\Omega_R$, continuous on
$\overline\Omega_R$, and there exists an $R_\epsilon>0$ such that
$|f_\epsilon(z)|\leq M$ for all
$z\in \partial\Omega_R$, provided that $R\geq R_\epsilon$.
Consequently, by the maximum principle, $|f_\epsilon(z)|\leq M$ for
all $z\in \Omega_R$ and since this argument works for any $R\geq
R_\epsilon$,
we see that in fact $|f_\epsilon(z)|\leq M$ for all $z\in
\Omega$. This yields
\[ |f(z)|\leq e^{\epsilon |z|^{\frac34}}|f_\epsilon(z)|\leq
  Me^{\epsilon |z|^{\frac34}} \]
for any $z\in \Omega$ and any $\epsilon>0$ and upon letting $\epsilon\to 0$, we obtain the
desired bound.
\end{proof}

\subsection{Asymptotics of difference equations}

Another important building block in the proof of mode stability is the
asymptotic behavior of solutions to difference equations.

\begin{theorem}[Poincar\'e]
  \label{thm:Poincare}
  Let $p,q: \N\to\C$ and suppose that
  \[ p_\infty:=\lim_{n\to\infty}p(n),\qquad
    q_\infty:=\lim_{n\to\infty}q(n) \]
  exist. Assume further that there exist $z_1, z_2\in \C$
  with $|z_1|>|z_2|$ and such that
  \[ z_j^2+p_\infty z_j+q_\infty=0,\qquad j\in \{1,2\}. \]
  Let $a: \N\to\C$ satisfy
  \begin{equation}
    \label{eq:Poincarediff}
    a(n+2)+p(n)a(n+1)+q(n)a(n)=0
    \end{equation}
  for all $n\in\N$. Then either there exists an $n_0\in\N$ such that
  $a(n)=0$ for all $n\geq n_0$ or we have
  \[ \lim_{n\to\infty}\frac{a(n+1)}{a(n)}\in\{z_1,z_2\}. \]
\end{theorem}

\begin{proof}[Idea of proof]
  In order to understand what is going on, we consider the
  \emph{limiting equation}
  \begin{equation}
    \label{eq:limdiff}
    a(n+2)+p_\infty a(n+1)+q_\infty a(n)=0.
    \end{equation}
  Then it follows that the functions $n\mapsto z_j^n$ for $j\in
  \{1,2\}$ solve this equation simply because
  \[ z_j^{n+2}+p_\infty z_j^{n+1}+q_\infty z_j^n=z_j^n(z_j^2+p_\infty
    z_j+q_\infty)=0. \]
  Consequently, the general solution of Eq.~\eqref{eq:limdiff} is
  given by $a(n)=\alpha_1 z_1^n+\alpha_2 z_2^n$, where $\alpha_j\in
  \C$ can be chosen arbitrarily. Thus, if $\alpha_1\not=0$, we can
  write
  \[ a(n)=\alpha_1 z_1^n \left
      [1+\frac{\alpha_2}{\alpha_1}\left(\frac{z_2}{z_1}\right)^n\right] \]
  and $\lim_{n\to\infty}\frac{a(n+1)}{a(n)}=z_1$ follows immediately because
  $|\frac{z_2}{z_1}|<1$ by assumption. On the other hand, if
  $\alpha_1=0$, we obviously have
  $\lim_{n\to\infty}\frac{a(n+1)}{a(n)}=z_2$.
  Thus, since $p(n)$ and $q(n)$ get arbitrarily close to $p_\infty$
  and $q_\infty$ for large $n$, the proof consists of showing that the above logic is stable
  under a suitable perturbation argument, see e.g.~\cite{Ela05}.
\end{proof}

\bibliography{specss}
\bibliographystyle{plain}

\end{document}